\documentclass[a4paper,12pt,oneside,reqno]{amsart}
\usepackage{hyperref}
\usepackage[headinclude,DIV13]{typearea}
\areaset{15.1cm}{25.0cm}
\usepackage{amsfonts,amssymb,amsmath,amsthm,bbm}
\usepackage[utf8]{inputenc}
\usepackage{graphicx, psfrag}

\newtheorem{theorem}{Theorem}[section]
\newtheorem{lemma}[theorem]{Lemma}
\newtheorem{proposition}[theorem]{Proposition}
\newtheorem{corollary}[theorem]{Corollary}

\theoremstyle{definition}

\theoremstyle{remark}

\def\paragraph#1{\noindent \textbf{#1}}

\numberwithin{equation}{section}

\def\<{\langle}
\def\>{\rangle}
\def\l{\lambda}

\def\th{\theta}

\def\R{{\mathbb R}}  %%
\def\N{{\mathbb N}}  %%
\def\P{{\mathbb P}}  %% carateri piu belle per campi di nombre
\def\E{{\mathbb E}}  %%
\def\T{{\mathbb T}}

\let\cal=\mathcal

 \def \ba {\begin{array}}
 \def \ea {\end{array}}

 \def \cR {{\cal R}}

 \newcommand{\be}{\begin{equation}}
 \newcommand{\ee}{\end{equation}}

\newcommand{\bea}{\begin{eqnarray}}
 \newcommand{\eea}{\end{eqnarray}}
\def\TH(#1){\label{#1}}\def\thv(#1){\ref{#1}}
\def\Eq(#1){\label{#1}}\def\eqv(#1){(\ref{#1})}

 \def \1{\mathbbm{1}}
\newcommand\st{S_N}
\newcommand\sti{\tilde{S}_N}
\numberwithin{equation}{section}

\begin{document}
\title{Convergence of  a kinetic equation to
a fractional diffusion equation}
\author[G. Basile]{Giada Basile}
 \address{G. Basile\\Institut f\"ur Angewandte Mathematik\\Rheinische
   Friedrich-Wilhelms-Uni\-ver\-si\-t\"at Bonn\\Endenicher Allee 60\\ 53115
   Bonn,
Germany}
\email{basile@uni-bonn.de}
\author[A. Bovier]{Anton Bovier}
 \address{A. Bovier\\Institut f\"ur Angewandte Mathematik\\Rheinische
   Friedrich-Wilhelms-Uni\-ver\-si\-t\"at Bonn\\Endenicher Allee 60\\ 53115
   Bonn,
Germany}
\email{bovier@uni-bonn.de}


\date{\today}
\begin{abstract}
A linear Boltzmann equation is interpreted as
%
%We consider a linear Boltzmann equation, which describes the evolution
%of the energy density distribution of the normal modes (or phonons)
%of the system. It can be interpreted as
%
the forward equation for the probability density of a Markov process
$(K(t), Y(t))$ on $(\T\times\R)$, where $\T$ is the one-dimensional
torus.  $K(t)$ is a autonomous reversible jump process, with waiting
times between two jumps with finite expectation value but infinite
variance.  $Y(t)$ is an additive functional of $K$, defined as
$\int_0^t v(K(s))ds$, where $|v|\sim 1$ for small $k$ .  We prove that
the rescaled process $N^{-2/3}Y(Nt)$ converge in distribution to a
symmetric L\'evy process, stable with index $\alpha=3/2$.
\end{abstract}
\thanks{This research was supported in part  through the German Research
Council in the SFB 611}

\maketitle

%------------comments--------

\section{Introduction.}
The understanding of
%macroscopic
thermal conductance in both classical and quantum mechanical systems
is one of the fundamental problems of non-equilibrium statistical
mechanics. A particular aspect that has attracted much interest is the
observation that autonomous translation invariant systems in
dimensions one and two exhibit anomalously large conductivity.  The
canonical example here is a chain of anharmonic oscillators introduced
by Fermi-Pasta-Ulam (FPU)\cite{FPU}, for which numerical evidence
shows a super-diffusive spreading of energy (see \cite{LLP} for a
general review).
% Recent experiments on carbon nano-tubes show the same features.
However, the rigorous analysis of energy transport mechanism presents
serious mathematical difficulties and few results are obtained
starting from microscopic dynamics.

The canonical approach to this problem, starting with the pioneering
work of Peierls \cite{Pei} for the case of weak non-linearity, is to
derive a Boltzmann-type equation that will describe the energy
transport in a kinetic limit.
%Lattice vibrations, responsable of heat transport,
%can be  described as a dilute gas of \emph{phonons}, and the
%evolution of its density is governed by a Boltzmann equation.
Recently, this approach was carried out rigorously for weakly
anharmonic FPU chains \cite{Pe,ALS,LS}. A linear Boltzmann equation
was derived in \cite{LSrm} for the harmonic chain with random masses.
 The same linear Boltzmann equation appears also as limit
of a random Schr\"odinger equation (see for example \cite{EY},
\cite{E}, \cite{Sp}, \cite{BPR}).

In \cite{BOS} a kinetic limit was performed for a system of harmonic
oscillators perturbed by a conservative stochastic noise and the
following linear Boltzmann equation is deduced for the the energy
density distribution of the normal modes, or phonons, characterized by
a wave-number $k\in[-1/2, 1/2]$:
\begin{equation}
\label{Be}
\partial_t W(t,u,k)+v(k)\partial_u W(t,u,k)=\int_\T dk'\; R(k,k')[W(t,u,k')-W(t,u,k)].
\end{equation}
The exact form of the scattering kernel $R$ and the velocity $v$
will be given below. The crucial features are, however, that the
kernel $R$ behaves like $k^{2}$ for small $k$, $\forall k'$, and like
${k'}^{2}$ for small $k'$, $\forall k$, while $|v(k)|\rightarrow 1$,
as $|k|\downarrow 0$. This conforms to the intuitive picture that
phonons with wave number $k$ travel with a velocity $v(k)$ and are
scattered with a rate $R(k,k')$.
%The aim was to understand the features % the reasons
%of the anomalous thermal conductivity
%predicted and observed for low dimensional systems,
%looking at the behaviour  of the wave modes of the
%lattice, or \emph{phonons}.
It is well known that super-diffusive spreading of energy is connected
with the fact that the mean free path of phonons with small wave
number $k$ has a macroscopic length (ballistic transport), which
follows essentially from the smallness of the rate with which these
phonons are scattered, together with the fact that they travel with
finite velocity. In fact, in \cite{BOS} it is proved that this system
exhibits anomalous conductance.

To analyse the Boltzmann equation (\ref{Be}) is to exploit the fact that
it can be reinterpreted as the forward equation for the probability density
of a Markov process
 $(K(t), Y(t))$ on $(\T
 \times\R)$.  Here $K(t)$ is a reversible jump process with rate  $R$
 and $Y(t)$ is an additive functional of $K$, given  as
 $Y(t)=\int_0^tds\;v(K(s))$. In a phononic picture, the process $Y(t)$
 describes the trajectory of a phonon.  For system with   diffusive
 energy spreading, one expects that the law of the rescaled process
 $Y(Nt)/\sqrt{N}$ converges to the solution of heat equation. On the
 other hand, we expect that this is not true for systems exhibiting
 ballistic transport.

 To understand heuristically what is to be expected, it is convenient
 to introduce a discrete time Markov chain, $X_i$, that records the
 sequence of values assumed by the continuous time chain $K(t)$ and
 the holding times $\tau(X_i)$, i.e. the time the chain $K(t)$ remains
 in the state visited in step number $i$. Then the process $Y(t)$, at
 the time of the $n$-th jump of $K(t)$, can be written as
 $S_n=\sum_{i=1}^{n} \tau(X_i)v(X_i)$.  We will see later that in our
 case, due to the fact that the transition kernel behaves as ${k'}^2$
 for small $k'$ with respect to the second argument, the stationary
 distribution of the chain $X_i$ is of the form $\pi(dk)=\phi(k)dk$,
 where $\phi(k)\sim k^2$ for $k<<1$. On the other hand, the
 distribution of the holding time $\tau(k)$ is of the form
 $\P(\tau(k)>s)\sim e^{-sk^2}$. Hence, in the stationary distribution,
 we have that
\begin{equation}\label{heu.1}
\P\big(\tau(X_i)>s\big) \sim \int \pi(dk) e^{-sk^2} \sim s^{-3/2}.
\end{equation}
Thus, since $v$ is antisymmetric and $|v(0)|=1$, we expect
$\tau(X_i)v(X_i)$ to be in the domain of attraction of a stable law of
index $\alpha=3/2$. Then, if the $\tau(X_i)v(X_i)$ were independent
random variables, $n^{-1/\alpha}S_n$ would converge to the
corresponding stable law and $n^{-1/\alpha}S_{[ns]}$ to a stable
L\'evy process.

The corresponding problems of the convergence of dependent random variables
 with heavy tail distributions to a stable process has been studied extensively in the literature. A general,  and very efficient,
 approach is explained e.g. in Durrett and
 Resnick \cite{DR}: first one uses methods from the theory of extremes of
dependent random variables to study the convergence of the point process of
scaled summands to a Poisson point process.
Then one writes the sum as an integral with respect to the point process
 and  uses a moment method to
show that this integral converges to the corresponding integral with respect
to the Poisson point process, which is a L\'evy process.

The conditions needed to establish such a result are mainly required
to assure the convergence to the Poisson process. Durrett and Resnick
\cite{DR} express these in terms of the asymptotic behaviour of the
conditional distribution.  Davis \cite{Da} considers stationary
processes and the corresponding well known mixing conditions (see
\cite{LLR}), while in \cite{JK} and \cite{DJ} a non stationary
generalisation is done, requiring some mixing conditions. Further
results are to be found in e.g.  \cite{DH,J1,J2}.

 In this paper, we will follow the general strategy outlined above and
 use a criterion for Poisson convergence that can be found in in
 \cite{BK} and which has the advantage of being rather easily verified
 %in our context.
 for an ergodic Markov chain. It requires asymptotic factorisation
 conditions for probabilities which hold on average.  This will allow
 us to prove convergence first of
 the sum $n^{-1/\alpha}S_{[nt]}$ to an $3/2$-stable L\'evy process in
 the $J_1$-Skorokhod topology. Since $S_n =Y(T_n)$, where
 $T_n=\sum_{i=1}^n\tau(X_i)$, and since $Y(t)$ is the piecewise linear
 interpolation of this function at the random sequence to times $T_n$,
 once we will have shown that $T_{[nt]}/n$ converges to $t$, we will
 obtain that
% where $n(t)$ is a non decreasing
% right-continuous function with integer values.
%
 the rescaled process $n^{-1/\alpha}Y(nt)$
 converges to the same limit process in the
 Skorokhod $M_1$-topology, which is the appropriate topology for the
 convergence of continuous process to a process with c\`adl\`ag paths.

%In \cite{JKO}
%the authors consider the same model, defined by \ref{Be}.
%They use the martingale approximation approach for an additive
%functional of a Markov chain and a coupling techniques to get the
%convergence of rescaled sum $S_n$ to an $\alpha-stable$ L\'evy
%process. Convergence of $Y(t)$ is not proved.

Let us note that Jara et al \cite{JKO} prove
similar convergence results for additive
 functionals of Markov processes using Martingale and coupling
 methods and apply their
 results to the same model we consider here.  Their methods
 do, however, only yield convergence of finite dimensional marginals.
 The methods we use here seem more straightforward  and direct, and give
 stronger results.

 Finally, using convergence of $n^{-1/\alpha}Y(nt)$ to an
 $\alpha$-stable L\'evy process, $\alpha\in(1,2)$, one can prove that
 the rescaled solution of the linear Boltzmann equation (\ref{Be})
%$W(Nt,N^{1/\alpha}u, k)$
converges to the solution of the following
 fractional diffusion equation
\begin{equation*}
\partial_t\bar{W}(t,u)=-(-\partial_u^2)^{\alpha/2}\bar{W}(t,u).
\end{equation*}
We refer to \cite{JKO} for the proof. Convergence of
the rescaled solution of a linear Boltzmann equation to the solution
of a fractional diffusion equation was independently proved by Mellet
et al \cite{MMM}, with purely analytical techniques.

\section{The model}
We consider the process $(K(t), Y(t))$ described by equation
(\ref{Be}). Denoting by $\T$ the one-dimensional torus, we choose
$v:\T\to\R$ and $R: \T\times\T\to\R_+$ as in \cite{BOS}, namely
\begin{equation}\label{v}
v(k)  = \frac{\sin(2\pi k)}{|\sin(\pi k)|},
\end{equation}
\begin{eqnarray}
\label{def:kernel}\nonumber
R(k,k')  =\frac{4}{3}\left[2\sin^2(2\pi k)\sin^2(\pi k')+
2\sin^2(2\pi k')\sin^2(\pi k) -\sin^2(2\pi k)\sin^2(2\pi k')\right].\\
\end{eqnarray}
Observe that the rate kernel $R$ is symmetric, not negative and it is
equal to zero only if $k=0$ or $k'=0$.  We remark that despite the
special case we consider, results depend essentially on the behaviour
of $v$ and $\R$ for small $k$, i.e.  for $k\ll 1$ $|v(k)|\sim 1$ and
$R(k,k')\sim k^{-2}$, $\forall k'\in\T/\{0\}$, and $R(k,k')\sim
k'^{-2}$ for $k\ll 1$, $\forall k\in\T/\{0\}$.
The jump process $K(t)$
%, with value in $\T$
is determined by the
generator
\begin{equation}
Lf(k)=\int_\T dk'\; R(k,k')[f(k')-f(k)],\hspace{0.4cm}\forall
f\in B(\T).
\end{equation}
The process $Y(t)$ is an additive functional of $K(t)$, defined as
$Y(t)=\int_0^t ds\; v(K(s))$.

%Let us focus on the jump process $K(t)$.
Disregarding the time, the stochastic sequence $\{X_i\}_{i\geq 0}$ of
states visited by $K(t)$ is a Markov chain with value in $\T$, with a
probability kernel $P$ concentrated on $\T$ given by
\begin{equation}
P(k,dk')=\phi(k)^{-1}R(k,k')dk'
\end{equation}
 where $\phi:\T\to\R_+$ is given by
\begin{equation}
\label{def:phi}
\phi(k)=\int_\T dk'\;R(k,k')
=\frac{4}{3}\sin^2(\pi k)(1+2\cos^2(\pi k)).
\end{equation}
% We will denote with
%$p(k,k')=\phi(k)^{-1}R(k,k')$ the density, which has the following form:
%\begin{equation}\label{prob}\begin{split}
%&p(k,k')
%%=\phi(k)^{-1}R(k,k')\\ &
%=
%8\frac{\cos^2(\pi k)}{1+2\cos^2(\pi k)}\sin^4(\pi k')+
%8\frac{\sin^2(\pi k)}{1+2\cos^2(\pi k)}\sin^2(\pi k')\cos^2(\pi k').
%\end{split}\end{equation}
We denote with $P^m$, $m\geq 2$, the $m$-th
convolution integral of $P$.
%
%Of course $P^m(k,dk')=p^m(k,k')dk'$,
%where $p^m(k,k')dk'$ is
%defined by the recoursive formula
%$$p^{m+1}(k,k')=p^m(k,\cdot)\circ p(\cdot,k').$$
Since the probability kernel $P$ is regular and strictly positive and
defined on a compact set, it is ergodic, i.e. there exists a strictly
positive probability distribution $\pi$ such that
%for every initial measure $\gamma:\T\to [0,1]$
%
$
\forall k\in\T$, $P^n(k,\cdot)\to \pi(\cdot)$, weakly, as $n\uparrow\infty$.
The stationary measure $\pi$ is given by $\pi(dk)=\phi(k)dk$.

We define two functions of the Markov chain $\{X_i\}_{i\geq 0}$: the
clock process, $T_n\in\R^+$, and the position process, $S_n\in\R$, by
\begin{equation}\label{def:T}
T_n=  \sum_{i=0}^{n-1}e_i[\phi(X_i)]^{-1},
\end{equation}
and
\begin{equation}\label{def:S}
S_n= \sum_{i=0}^{n-1}e_i[\phi(X_i)]^{-1}v(X_i).
\end{equation}
Here $\{e_i\}_{i\geq 0}$ are i.i.d. exponential random variables with
parameter $1$.  The clock process, $T_n$, is the time of the $n$-the jump of
the
process $K(t)$.  It is a sum of positive random
variables with finite expectation, as one can easily check using the
explicit form of the probability density (see Eq.
%(\ref{prob}),
(\ref{form:prob}) below).  The position process, $S_n$, is  the
value of the position of  $Y(t)$ at time $T_n$, i.e. $S_n=Y(T_n)$.
It is a sum of real variables with zero mean and infinite
variance. More precisely, for any $i\in \N$, for large $\lambda$
\begin{equation}\label{htd}
\P\Big[\big|e_i\phi(X_i)^{-1}v(X_i)\big|> \lambda\Big]\sim\lambda^{-3/2},
\end{equation}

%%%%%%%%%%%%%%%%%%%%%%%%%%%%%%%%%%%%%%%%%%

Let  $T^{-1}$ denote the right-continuous inverse of $T_n$, i.e. let
\begin{equation}\label{def:T-1}
T^{-1}(t)\equiv \inf\{n:T_n\geq t\}.
\end{equation}
We can
represent the original processes, $(K(t), Y(t))$, as follows:
\begin{equation}\label{K,Y}
\begin{split}
K(t)= & X_{[T^{-1}(t)-1]}\\
Y(t)= & S_{[T^{-1}(t)-1]} + v(X_{[T^{-1}(t)-1]})(t-T_{[T^{-1}(t)-1]}).
\end{split}.
\end{equation}
In particular,  $Y(t)$  is the function defined by
linear interpolation between its values $S_{n}$ at the random points
$T_n$  (we take  $S_0=0$).

%%%%%%%%%%%%%%%%%%%%%%%%%%%%%%%%%%%%%%%%%%%%%%%

\section{Main results. }
We assume that the initial distribution, $\mu$, of the process $X$
satisfies the condition
\begin{equation}\label{ass:mu}
\int_T d\mu(k) k^{-2}<\infty,
\end{equation}
which guarantees in particular that $\E_\mu\left[e_0\phi(X_0)^{-1}\right]
<\infty$.

We define  the rescaled processes
\begin{equation}\label{def:TN.SN}\begin{split}
T_N(\theta)=\frac{1}{N}T_{\lfloor N\theta \rfloor},\hspace{0.4cm}
S_N(\theta)=\frac{1}{N^{2/3}}S_{\lfloor N\theta \rfloor},\hspace{0.4cm}
T_N^{-1}(\theta)=\frac{1}{N}T^{-1}(N\theta),
\end{split}\end{equation}
where $\lfloor \cdot\rfloor$ denotes the lower integer part of
$\cdot$.  Since $T_n$ is a sum of positive variables with finite
expectation, we expect that both $T_N(\theta)$ and $T_N^{-1}(\theta)$
converge in probability (and thus in distribution) to $\theta$, in the
topology of uniform convergence on compact intervals.  This will be
proved in Proposition \ref{conv:TN}.

On the other hand, $S_n$ is a sum of centred random variables whose
tail behaviour is given in (\ref{htd}). Thus we expect that the
rescaled process $S_N$ converges to a stable process with index $3/2$.
This is the content of the following theorem.

\begin{theorem}\label{theo:conv}
  Let $S_N$ be the process defined in $(\ref{def:TN.SN})$.  Then for
  any $0<\mathcal{T}<\infty$, the process $\{S_N(\theta)\}_{0\leq
    \theta\leq \mathcal {T}}$ converges to $\{V(\theta)\}_{0\leq
    \theta\leq\mathcal{T}}$, where $V$ is a symmetric
L\'evy process stable with index $3/2$.
%with zero drift and L\'evy measure $\nu(dx)=|x|^{-5/2} dx$.
  Convergence is in distribution on the Skorokhod space of c\`adl\`ag
  functions equipped with the $J_1-topology$.
\end{theorem}
Combining this theorem with Proposition \ref{conv:TN},  we will prove
that $S_N\circ
T_N^{-1}$ converges in distribution to $V$. This will imply our  main
theorem.
%
%---------MAIN THEOREM--------------------
%
\begin{theorem}\label{theo:conv2}
Let $S_N$, $T_N^{-1}$ be the processes defined in (\ref{def:TN.SN}).
For every $0<\mathcal{T}<\infty$, the process
$
\{S_N(T_N^{-1}(\theta))\}_{0\leq \theta\leq\mathcal{T}}$, converges
to $\{V(\theta)\}_{0\leq \theta\leq\mathcal{T}}$,
where $V$ is a symmetric L\'evy process stable with index $3/2$.
Convergence is in distribution
on the Skorokhod space of c\`adl\`ag functions equipped with the
$J_1-topology$.

Moreover, for every $N\geq 1$ the process
$Y_N(t)=\frac{1}{N^{2/3}}\int_0^{Nt}ds\; v(K_s)$ is the function
defined by linear interpolation between its values $S_N(\theta)$ at
points $T_N(\theta)$, with $\theta\in[0,T_N^{-1}(\mathcal{T})]$.
In particular,  $Y_N$ converges to $V$ in
distribution in the $M_1$-Skorokhod  topology.
\end{theorem}

%\begin{remark} Jara et al  \cite{JKO} prove similar convergence results
%for additive functionals of Markov processes using different methods.
%Their methods do, however, only yield convergence of finite dimensional
%marginals. Our methods are more direct, following the construction of L\'evy
%processes as presented by It\^o \cite{Ito} combined with methods from the
%theory of extremes of dependent random variables, and yield the
%tightness results naturally.
%\end{remark}

\section{Sketch of the proof of Theorem \ref{theo:conv}}

In this section we present the key steps of the proof of Theorem 1.
The technical details will be given in Sections 5, 6, and 7.
As we mentioned in the introduction, we follow the strategy of
considering the sequence of the point processes associated to $S_N$.
At the first step we define this sequence of point processes and we
show that it converge to a Poisson point process.  Then we prove that
the limit process for $S_N$ exists and it is a L\'evy process stable
with index $3/2$. Finally, we prove the tightness for the sequence
$S_N$.
%%%%%%%%%%%%%%%%%%%%%%%%%%%%%%%%%%%%%%%%%%%%%%%%%%%%%%%%%%%
%
\subsection{Point processes.}
Define the real valued random variables
\begin{equation}\label{def:psi}
\psi_n\equiv \phi(X_n)^{-1}v(X_n),\hspace{0.4cm}n\in\N_0.
\end{equation}

For some fixed  $c>0$,
we decompose  $\st(\theta)$ into two parts,
$
\st(\theta)=\st^>(\theta)+\st^<(\theta),
$
where
\begin{equation}\label{def:S+S-}\begin{split}
\st^>(\theta)=N^{-2/3}\sum_{n=0}^{[N\theta]-1}e_n\psi_n\1_{\{e_n|\psi_n|>cN^{2/3}\}}
\\
\st^<(\theta)=N^{-2/3}\sum_{n=0}^{[N\theta]-1}e_n\psi_n\1_{\{e_n|\psi_n|\leq
cN^{2/3}\}}.
\end{split}\end{equation}
We will see later that
$\st^<(\theta)$ vanishes  as  $N\to\infty$  and then   $c\to 0$.
More precisely:
%-------------LEMMA S-----------------------
%
\begin{lemma}\label{proof:S-}
Let $\st^<(\theta)$, $\theta\in[0,\mathcal{T}]$ be the process defined in
(\ref{def:S+S-}). Then
for every $\theta\in[0,\mathcal{T}]$
\begin{equation}\label{S-}
\E\left[\big|\st^<(\theta)\big|^2\right]\leq C_0 \theta \sqrt c
+ C_1 \theta N^{-1/3},
\end{equation}
where $C_0, C_1 <\infty$ are positive constant.
\end{lemma}
%-----------------------------------------------
We will prove this lemma in the next section.

On the other hand, $S_N^>$ will be connected to two Poisson processes.

We split  the sum $\st^>(\theta)$ into two parts:
\begin{eqnarray}\label{split.1}\nonumber
\st^>(\theta)&=&
     \frac{1}{N^{2/3}}\sum_{n=0}^{[N\theta]-1}e_n\psi_n\;
       \1_{\{e_n\psi_n>c N^{2/3}\}}\\
    &+&\frac{1}{N^{2/3}}\sum_{n=0}^{[N\theta]-1}e_n\psi_n\;
       \1_{\{e_n\psi_n<-c N^{2/3}\}}.
   \end{eqnarray}
Defining the random variables
\begin{equation}\label{def:X+X-}
X_{i,N}^+=\frac{e_i\psi_i}{N^{2/3}}\;\1_{\{\psi_i\geq0\}},\hspace{1cm}
X_{i,N}^-=\frac{e_i\psi_i}{N^{2/3}}\;\1_{\{\psi_i\leq0\}},
\end{equation}
with values in $\R_+$, $\R_-$, respectively,
and the associated point processes
 $\mathcal{R}_N^+$ and $\mathcal{R}_N^-$,
\begin{equation}\label{def:RN}
\mathcal{R}_N^+=\sum_{i\in \N_0}\delta_{i/N,X_{i,N}^+},
\hspace{0.4cm}\mathcal{R}_N^-=\sum_{i\in \N_0}\delta_{i/N,X_{i,N}^-}.
\end{equation}
we can rewrite $\st^>(\theta)$ as
\begin{equation}\label{split.2}
\st^>(\theta)=\int_0^\theta  \int_c^\infty
x\; \cR^+_N(ds,dx)+
\int_0^\theta \int_{-\infty}^{-c}x\; \cR_N^-(ds,dx)
\end{equation}
The following  Proposition states that the two point processes
$\mathcal{R}_N^+$ and $\mathcal{R}_N^+$ converge to Poisson point processes.

\begin{proposition}\label{prop:PPP}
Let $\mathcal{R}_N^+$, $\mathcal{R}_N^-$ be the point processes
defined in (\ref{def:RN}). Then
the point process $\mathcal{R}_N^\pm$
converges in distribution to  Poisson point processes $\mathcal{R}^\pm$ on
$\R_+\times\R_+$, resp. $\R_+\times\R_-$, with intensity
measures $dt\times d\nu^\pm(x)\equiv\frac{3}{2}a
|x|^{- 5/2}dx$, with $a>0$ some explicit constant.
\end{proposition}
We will prove this proposition in Section 6.
%
%
%%%%%%%%%%%%%%%%%%%%%%%%%%%%%%%%%%%%%%%%%%%%%%%%%%%%%%%%%%%%%%%%%%%%%%%
\subsection{Limit process for $S_N$}
We define a process $V$ on $\R_+$  with
values in $\R$ by
\begin{equation}\label{def:V}
V(\theta)=\int_0^\theta\int_0^\infty
\mathcal{R}^+(ds,dx)x
+\int_{0}^\theta\int_{-\infty}^{0} \mathcal{R}^-(ds,dx)x,
\end{equation}
and
\begin{equation}\label{def:Vc}
V^>_c(\theta)=\int_0^\theta\int_c^\infty
\mathcal{R}^+(ds,dx)x
+\int_{0}^\theta\int_{-\infty}^{-c} \mathcal{R}^-(ds,dx)x,
\end{equation}
where $\mathcal{R}_+$, $\mathcal{R}_-$ are the Poisson point processes
defined in Proposition \ref{prop:PPP}.

\begin{proposition} \label{levy-ito}
  Let be $V$, the process defined in (\ref{def:V}).
 $V$ is well defined and it is symmetric a L\'evy
  process stable of  index $3/2$, with Laplace functional
$\Psi(\l)=C_0 |\l|^{3/2}$, where $C_0=a\int dy|y|^{-5/2}(1-\cos y)$.

Moreover, $V_c\rightarrow V$ as $c\downarrow 0$, in the topology of uniform
convergence on compact intervals.
\end{proposition}

\begin{proof}
  We split $V(\theta)$ into two parts
  $V(\theta)=V^>_c(\theta)+V^<_c(\theta)$.  Clearly, the two
  processes, if they exist, are independent.  Since, for any compact
  interval, $I$, the total intensity of $I\times (c,\infty)$ equals
\begin{equation}
 \int_I ds\int_c^\infty d\nu^+(x)= |I| a c^{-3/2}<\infty
\end{equation}
the process $V_c^>(\theta)$ is a finite sum, almost surely.

Moreover, by direct computation and using the fact that
$\E[V_c^<(\theta)]=0$, we get
\begin{equation}
\E\left[|V_c^<(\theta)|^2\right]=\int_0^\theta ds\;\int_0^c d\nu^+(x)x^2+
\int_0^\theta ds\;\int_{-c}^0 d\nu^-(x)x^2=6\theta a c^{1/2},
\end{equation}
hence also $V^<(\theta)$ is almost surely finite, and tends to zero in probability, as $c\downarrow 0$.

Then $V(\theta)$ is almost surely finite, and since it is the sum of
two independent processes which are right-continuous and have
independents increments, it satisfies the hypothesis of a L\'evy
process, and it is full characterized by the one-dimensional
distribution.  This is uniquely determined by its characteristic
exponent $\Psi:\R\to\mathbb{C}$ defined as
\begin{equation}
e^{-t\Psi(\lambda)}=\E\left[\exp\{i\lambda V(\theta)\}\right]
\end{equation}
which by direct computation is given by
\begin{equation}
\Psi(\lambda)=\int_{-\infty}^\infty d\nu(x)\left(1-e^{i\lambda x}\right),
\end{equation}
with the L\'evy measure $\nu$ defined on $\R\setminus \{0\}$ given by
$d\nu(x)=a |x|^{-5/2}dx$.
For every $\lambda\in\R$
\begin{equation}
\Psi(\lambda)=|\lambda|^{3/2} a\int_\R dy\; |y|^{-5/2}(1-\cos y)=C_0|\lambda|^{3/2},
\end{equation}
which is the characteristic exponent
of a symmetric L\'evy process stable with index $3/2$.
\end{proof}
%%%%%%%%%%%%%%%%%%%%%

Proposition \ref{levy-ito} together with Lemma \ref{proof:S-} and
Proposition \ref{prop:PPP} implies convergence of the finite dimensional
distributions of $\st$ to $V$. Since $V$ has stationary and independent
increments, it is in fact enough to prove the convergence of the
 one-dimensional distributions.

\begin{corollary}\label{finite}
For any $\th\in \R_+$,
\begin{equation}
S_N(\th)\rightarrow V(\th),
\end{equation}
in distribution.
\end{corollary}

\begin{proof} From the representation (\ref{def:RN}) and Proposition
\ref{prop:PPP}, and since the intensity of $[0,\th]\times (c,\infty)$
is finite, it follows readily that
\begin{equation}
\label{levy-ito.1}
S_N^> (\th)\rightarrow V_c(\th), \quad\hbox{\rm {as} }N\uparrow \infty,
\end{equation}
in distribution. Moreover, $V_c(\th)\rightarrow V(\th)$, as $c\downarrow 0$,
 and
\begin{equation}\label{levy-ito.3}
\lim_{c\downarrow0}\limsup_{ N\uparrow \infty} |S_N^<(\th)|=0,
\end{equation}
in probability, by Lemma \ref{proof:S-}.
This implies the assertion of the corollary.
\end{proof}

To conclude the proof of Theorem 1, we  need to complement this corollary
with the proof of tightness of the sequence $S_N$. This will be postponed to
Section 6.

%%%%%%%%%%%%%%%%%%%%%%%%%%%%%%%%%%%%%%%%%%%

\section{Moment estimations.}
In this section we collect and prove some useful moment estimates.

We start with some preliminary results on  the transition probability density
 for the Markov
chain $\{X_n\}$. They are given by
\begin{eqnarray}\label{prob}
p(k,k')
&=&\phi(k)^{-1}R(k,k')\\\nonumber
&=&
8\frac{\cos^2(\pi k)}{1+2\cos^2(\pi k)}\sin^4(\pi k')+
8\frac{\sin^2(\pi k)}{1+2\cos^2(\pi k)}\sin^2(\pi k')\cos^2(\pi k').
\end{eqnarray}
We denote by $p^m(k,k')$ the $m$-step transition densities, i.e.
$$
p^2(k,k')=p(k,\cdot)\circ p(\cdot,k')\equiv \int_\T dk_1\;
p(k,k_1)p(k_1,k')
$$
and in  the same way,  for every $m\geq 1$
$$
p^{m+1}(k,k')=p^m(k,\cdot)\circ p(\cdot,k')=p(k,\cdot)\circ p^m(\cdot,k').
$$
Observe that $P^m(k,dk')=p^m(k,k')dk'$, where $P^m$ is the m-th
convolution integral of the probability kernel $P$.

In the next proposition we give an explicit formula for $p^m$.
\begin{proposition}
For every $m\geq1$, $p^m(k,k')$ has the following form
\begin{equation}\label{form:prob}\begin{split}
p^m(k,k')=&
8\frac{\cos^2(\pi k)}{1+2\cos^2(\pi k)}\left[a_m\sin^4(\pi k')+ b_m\sin^2(\pi
k')\cos^2(\pi k')\right]\\
+ & 8\frac{\sin^2(\pi k)}{1+2\cos^2(\pi k)}
\left[c_m\sin^4(\pi k')+ d_m\sin^2(\pi
k')\cos^2(\pi k')\right]
\end{split}\end{equation}
where
\begin{equation}\label{coeff}
\left\{
\begin{array}{ll}
a_1=d_1=1, \hspace{0.4cm}b_1=c_1=0, & m=1\\
0< a_m,b_m,c_m,d_m <1,         & m\geq 2
\end{array}\right.
\end{equation}
\end{proposition}
\begin{proof}
By direct computation,
\begin{equation}\begin{split}
p^2(k,k')=&
8\frac{\cos^2(\pi k)}{1+2\cos^2(\pi k)}\left[a_2\sin^4(\pi k')+ b_2\sin^2(\pi
k')\cos^2(\pi k')\right]\\
+ & 8\frac{\sin^2(\pi k)}{1+2\cos^2(\pi k)}
\left[c_2\sin^4(\pi k')+ d_2\sin^2(\pi
k')\cos^2(\pi k')\right]
\end{split}\end{equation}
with $a_2,b_2,c_2,d_2$  positive, $a_2=d_2$.
In the same way for $m\geq 2$ we find expression (\ref{form:prob}),
 where the  coefficients $a_m,b_m,c_m,d_m$ are  given by the following
recursive formula
\begin{equation}
\left(
\begin{array}{ll}
a_{m+1} & b_{m+1}\\
c_{m+1} & d_{m+1}
\end{array}
\right)
=
\left(
\begin{array}{ll}
a_2 & b_2\\
c_2 & d_2
\end{array}
\right)
\left(
\begin{array}{ll}
a_{m} & b_{m}\\
c_{m} & d_{m}
\end{array}
\right),\hspace{0.4cm}m\geq 2
\end{equation}
In particular  $a_m, b_m, c_m, d_m$ are   positive for every
$m\geq 2$ .
 Moreover, by condition $\int_\T dk'\; p^m(k,k')=1$, $\forall
m\geq 1$, we get
the relations
$$
3a_m+b_m=3,\hspace{1cm}3c_m+d_m=1
$$
which says that the coefficients are uniformly bounded,
i.e. $a_m,b_m,c_m,d_m < 1$ for every $m\geq 1$.
\end{proof}

%%%%%%%%%%%%%%%%%%%%%%%%%%%%%%%%%%%%%%%%%%%%%%%%%%
%%%%%%%%%%%%%

 We now  give the proof of Lemma \ref{proof:S-}.

\begin {proof} (Lemma \ref{proof:S-}). Let us write
\begin{equation}\label{S-1}\begin{split}
 &\E\left[\big|\st^<(\theta)\big|^2\right]=N^{-4/3}\sum_{n=0}^{\lfloor N\theta\rfloor-1}
\E\left[\big|e_n\psi_n\big|^2
\1_{\{e_n|\psi_n|\leq c N^{2/3}\}}\right]\\
&
+2N^{-4/3}\sum_{n=1}^{\lfloor N\theta\rfloor-1}\sum_{0\leq m<n}
\E\left[e_n\psi_n e_m\psi_m
\1_{\{e_n|\psi_n|\leq c N^{2/3}\}}
\1_{\{e_m|\psi_m|\leq c N^{2/3}\}}\right].
\end{split}\end{equation}
Let us focus on the first sum on the right hand side. For every $n\geq 1$
we have
\begin{equation}\label{Epsi2}\begin{split}
&\E\left[\big|e_n\psi_n\big|^2
\1_{\{e_n|\psi_n|\leq c N^{2/3}\}}\right]\\
&=\int_0^\infty dz\;e^{-z} z^2\int_\T d\mu(k_0)
\int_\T dk\;p^n(k_0,k)|\psi(k)|^2
\1_{\{z|\psi(k)|\leq c N^{2/3}\}},
\end{split}\end{equation}
where
\begin{equation}\label{mod_psi}
|\psi(k)|=|\phi(k)^{-1}v(k)|=
\frac{3|\cos(\pi k)|}{2\sin^2(\pi k)(1+2\cos^2(\pi k))}.
\end{equation}
Using the explicit form of $p^n$ given in (\ref{form:prob}), an
elementary computation reveals that
\begin{equation}
\E\left[\big|e_n\psi_n\big|^2
\1_{\{e_n|\psi_n|\leq c N^{2/3}\}}\right]\leq C_0 \sqrt c N^{1/3}
+C_1,
\end{equation}
with $C_0, C_1 <\infty$ are positive constants.
Thus
\begin{equation}\label{ESpsi2}
N^{-4/3}\sum_{n=0}^{\lfloor N\theta\rfloor -1}
\E\left[\big|e_n\psi_n\big|^2
\1_{\{e_n|\psi_n|\leq c N^{2/3}\}}\right]\leq C_0\theta\sqrt{c}
+ C_1 \theta N^{-1/3}.
\end{equation}

The second sum on the r.h.s. of Eq. (\ref{S-1}) is in fact equal to
zero. Namely, for all $ n>m\geq 1$ we have
\begin{eqnarray}\label{simm}
&&\E\left[e_n\psi_n e_m\psi_m
\1_{\{e_n|\psi_n|\leq c N^{2/3}\}}
\1_{\{e_m|\psi_m|\leq c N^{2/3}\}}\right]\\\nonumber
&&=
\int_0^\infty dz\; z e^{-z}\int_0^\infty du \; u e^{-u}
\int_\T d\mu(k_0)\int_\T dk\; p^m(k_0,k)\psi(k)
\1_{\{|\psi(k)|\leq c N^{2/3} z^{-1}\}}\\\nonumber
&&\quad\times
\int_\T dk\; p^{(n-m)}(k,k')\psi(k')
\1_{\{|\psi(k')|\leq c N^{2/3} u^{-1}\}}=0,
\end{eqnarray}
since $\psi(k)=-\psi(-k)$ and $p^\ell(k,k')=p^\ell(k,-k')$ $\forall k,
k'\in\T$, $\forall \ell\geq 1$.
This concludes the proof of Lemma \ref{proof:S-}.
\end{proof}
%-------------------------------
%------------------------------
The following related lemma will be needed in Section 8.

\begin{lemma}\label{mom1}
There exists $C<\infty$, such that for  $n\geq1$
%and for all $N$ large enough,
\begin{equation}
\E\left[e_n|\psi_n|\1_{\{e_i|\psi_n|>N^{2/3}\}}\right]\leq C N^{-1/3}.
\end{equation}
\end{lemma}

\begin{proof}
For every $n\geq 1$ we have
\begin{equation}\begin{split}
\E\left[e_n|\psi_n|\1_{\{e_n|\psi_n|>N^{2/3}\}}\right]
=
 \int_0^{\infty}dz\; z e^{-z}\int_\T d\mu(k_0)
\int_\T dk\;
 p^n(k_0,k)\big|\psi(k)\big|\1_{\{z|\psi(k)|>N^{2/3}\}},
\end{split}\end{equation}
and using the explicit formula (\ref{form:prob}) for $p^n$ one
easily finds that this expression is bounded by $C
N^{-1/3}$, with $C<\infty$.
\end{proof}

%---------------------------------
%---------------------------------
\begin{lemma}\label{ESS}
  Let $\sti^<(\theta)$, $\theta\in[0,\mathcal{T}]$, be the
  process defined in (\ref{def:sti}). Then for every
  $0\leq r\leq s\leq t\leq\mathcal{T}$
\begin{equation}
\E\left[|\sti^<(s)-\sti^<(r)|^2|\sti^<(t)-\sti^<(s)|^2\right]\leq A_0(t-r)^2,
\end{equation}
for some constant   $A_0<\infty$.
\end{lemma}

%-----------------------------------
\begin{proof}
For every $0\leq r< s< t\leq\mathcal{T}$ we have:
\begin{eqnarray}
&&|\sti^<(s)-\sti^<(r)|^2|\sti^<(t)-\sti^<(s)|^2\\\nonumber
&&=\frac{1}{N^{8/3}}
\sum_{i=\lfloor Nr\rfloor}^{\lfloor Ns\rfloor -1}
\sum_{j=\lfloor Nr\rfloor}^{\lfloor Ns\rfloor -1}
  e_i\psi_i\1_{\{e_i|\psi_i|\leq  N^{2/3}\}}
  e_j\psi_j\1_{\{e_j|\psi_j|\leq  N^{2/3}\}}\\
\nonumber&&\times
\sum_{h=\lfloor Ns\rfloor}^{\lfloor Nt\rfloor -1}
\sum_{l=\lfloor Ns\rfloor}^{\lfloor Nt\rfloor -1}
  e_h\psi_h\1_{\{e_h|\psi_h|\leq  N^{2/3}\}}
  e_l\psi_l\1_{\{e_l|\psi_l|\leq  N^{2/3}\}}.
\end{eqnarray}
By (\ref{simm}), the only terms with non-vanishing expectation are
\begin{equation}
\frac{1}{N^{8/3}}
\sum_{i=\lfloor Nr\rfloor}^{\lfloor Ns\rfloor -1}
  \big|e_i\psi_i\big|^2\1_{\{e_i|\psi_i|\leq  N^{2/3}\}}
\sum_{h=\lfloor Ns\rfloor}^{\lfloor Nt\rfloor -1}
  \big|e_h\psi_h\big|^2\1_{\{e_h|\psi_h|\leq  N^{2/3}\}}.
\end{equation}
For every $h>i$ we have
\begin{eqnarray}
&&\E\left[
    \big|e_i\psi_i\big|^2\1_{\{e_i|\psi_i|\leq  N^{2/3}\}}
    \big|e_h\psi_h\big|^2\1_{\{e_h|\psi_h|\leq  N^{2/3}\}}
  \right]
\\\nonumber
&&=
  \int_0^\infty dz\; z^2 e^{-z}
  \int_0^\infty du\; u^2 e^{-u}
  \int_\T d\mu(k_0)\\
  \nonumber && \quad \times\int_\T dk\; p^i(k_0,k)|\psi(k)|^2
           \1_{\{z|\psi(k)|\leq  N^{2/3}\}}
          \int_\T dk'\; p^{(h-i)}(k,k')|\psi(k')|^2
           \1_{\{u|\psi(k')|\leq  N^{2/3}\}}.
\end{eqnarray}
Using (\ref{form:prob}), we find that for  all $k\in\T$,
\begin{eqnarray}
&&\int_\T dk'\; p^{(h-i)}(k,k')|\psi(k')|^2
           \1_{\{u|\psi(k')|\leq  N^{2/3}\}}\\\nonumber
&&   \leq\frac{1}{\sqrt u}N^{1/3}
        B\left(\frac{\cos^2(\pi k)}{1+2\cos^2(\pi k)}
             +\frac{\sin^2(\pi k)}{1+2\cos^2(\pi k)}
        \right),
\end{eqnarray}
where $B$ is a finite constant.
Thus, $\forall k_0\in\T$
\begin{equation}\begin{split}
\int_\T dk\; p^i(k_0,k)|\psi(k)|^2
           \1_{\{z|\psi(k)|\leq  N^{2/3}\}}
          \int_\T dk'\; p^{(h-i)}(k,k')|\psi(k')|^2
           \1_{\{u|\psi(k')|\leq  N^{2/3}\}}\\
\leq \frac{1}{\sqrt z}\frac{1}{\sqrt u}N^{2/3}
        B_1\left(\frac{\cos^2(\pi k_0)}{1+2\cos^2(\pi k_0)}
             \frac{\sin^2(\pi k_0)}{1+2\cos^2(\pi k_0)}
        \right),
\end{split}\end{equation}
with $B_1<\infty $. Finally we get
\begin{eqnarray}\nonumber
\E\left[|\sti^<(s)-\sti^<(r)|^2|\sti^<(t)-\sti^<(s)|^2\right]
&\leq& \frac{1}{N^{8/3}}
          \sum_{i=\lfloor Nr\rfloor}^{\lfloor Ns\rfloor -1}
          \sum_{h=\lfloor Ns\rfloor}^{\lfloor Nt\rfloor -1}
           A_0 \; N^{2/3}\\
& \leq&  A_0(t-r)^2,
\end{eqnarray}
which is the assertion of the lemma.
\end{proof}

%%%%%%%%%%%%%%%%%%%%%%%%%%%%%%%%%%%%%%%%%%%%%%%%%%%%%%
%%%%%%%%%%%%%%%%%%%%%%%%%%%%%%%%%%%%
\section{Proof of Proposition \ref{prop:PPP}}\label{dim:PPP}
This section is devoted to verify conditions of the
Theorem 2.1 in \cite{BK}, which guarantees the convergence of the
point processes defined in (\ref{def:RN}) to Poisson point processes.
%A general extremal value theorem ((2.1) in \cite{BK}) gives some
%condition for the convergence of a point process to a Poisson point
%process.
We recall the statement of Theorem 2.1 of \cite{BK}.
\begin{theorem}\label{theorem1}[\cite{BK}]
Denote by $\sum_{\alpha_N(\ell)}$ the sum over all the ordered sequences of
different indices $(i_1,..,i_\ell)$ with values in $\{1,..,N\}$.
Let be $\{Z_{i,N}^+\}$ (resp. $\{Z_{i,N}^-\}$), $N\geq 1$,
$1\leq i\leq N$, an array of random variables with values in $\R_+$
(resp. in $\R_-$).
Assume that for every $\ell>0$ and all sets of positive constants
$(\tau_1,..,\tau_\ell)$
\begin{equation}
\sum_{\alpha_N(\ell)}\P\left[|Z_{i_1,N}^\pm|>\tau_1,...,
|Z_{i_\ell,N}^\pm|>\tau_\ell)\right]\to
a^\ell\frac{1}{\tau_1^{3/2}}...\frac{1}{\tau_\ell^{3/2}},
\hspace{0.4cm}\hbox{\rm as}\,\, N\to\infty,
\end{equation}
for some $a>0$. Then the point process
$\sum_{i=0}^{N-1}\delta_{i/N,Z_{i,N}^\pm}$
converges in distribution to a Poisson point process $\mathcal{R}^\pm$ on
$[0,\mathcal{T}]\times\R_\pm$ with intensity
measure $dt\times d\nu^\pm(x)$, where $d\nu^\pm(x)=\frac{3}{2}a
|x|^{-5/2}dx$.
\end{theorem}

Our goal is  to verify these conditions for the random variables
$\{X_{i,N}^+\}$, $\{X_{i,N}^-\}$ defined in (\ref{def:X+X-}).
In order to do so, we need  the following lemma.
\begin{lemma}\label{lemma1}
  For every $\ell>0$, for every sequence of different indices,
  $i_1,..,i_\ell\in\{1,..,N\}$, and all sets of positive constants,
  $(\tau_1,..,\tau_\ell)$, the following statements hold:
\begin{itemize}
\item[(i)]
if  $(i_{j+1}-i_j)\geq 2$ for every
$j=1,..,\ell-1$, then
\begin{equation}
\lim_{N\to\infty}\; N^\ell\tau_1^{3/2}...\tau_\ell^{3/2}\;
\P\left[e_1\big|\psi_{i_1}\big|>N^{2/3}\tau_1,...,
e_\ell\big|\psi_{i_\ell}\big|>N^{2/3}\tau_\ell\right]=\tilde{C_\ell},
\end{equation}
with $\tilde{C_\ell}>0$.
\item[(ii)]
If $(i_{j+1}-i_j)=1$ for some $j=1,..,\ell-1$, then
\begin{equation}
\lim_{N\to\infty}\; N^\ell\tau_1^{3/2}...\tau_\ell^{3/2}\;
\P\left[e_1\big|\psi_{i_1}\big|>N^{2/3}\tau_1,...,
e_\ell\big|\psi_{i_\ell}\big|>N^{2/3}\tau_\ell\right]=0.
\end{equation}
\end{itemize}
\end{lemma}

\begin{proof}
Let us consider
\begin{equation}\label{prob1}
\int_0^\infty dz\; e^{-z}
\int_\T dk_1\;
p^{m}(k,k')\1_{\{|\psi(k')|>N^{2/3}\tau z^{-1}\}},
\end{equation}
with $m\geq 1$.  Using (\ref{form:prob}) and the explicit expression
 for $\psi$ (\ref{mod_psi}), one easily finds that for $m\geq 2$ this
quantity is bounded by
\begin{equation}\label{in:prob1}
\begin{split}
%\lim_{N\to\infty}N
%\int_0^\infty dz\; e^{-z}
%\int_\T dk'\;
%p^{m}(k,k')\1_{\{|\psi(k')|>N^{2/3}\tau z^{-1}\}}\leq\\
C_1\left[ \frac{\cos^2(\pi k)}{1+2\cos^2(\pi k)}+
\frac{\sin^2(\pi k)}{1+2\cos^2(\pi k)}
\right]
\frac{1}{N}\frac{1}{\tau^{3/2}}
\end{split}
\end{equation}
with $C_1<\infty$. On the other hand, for $m=1$, we get
\begin{equation}\label{in:prob1b}\begin{split}
&\int_0^\infty dz\; e^{-z}
\int_\T dk'\;
p(k,k')\1_{\{|\psi(k')|>N^{2/3}\tau z^{-1}\}}\\
&\leq C_0\left[ \frac{\cos^2(\pi k)}{1+2\cos^2(\pi k)}+
\frac{\sin^2(\pi k)}{1+2\cos^2(\pi k)}
\right]\frac{1}{N^{5/3}}\frac{1}{\tau^{5/2}},
\end{split}\end{equation}
with $C_0<\infty$.

Next we consider
\begin{equation}\label{prob2}\begin{split}
&\int_0^\infty du\;e^{-u}\int_\T dk_1\;
p^{n}(k,k_1)\1_{\{|\psi(k_1)|>N^{2/3}\tau_1 u^{-1}\}}
\int_0^\infty dz\; e^{-z}\\
&\times
\int_\T dk'\;
p^{m}(k_1,k')\1_{\{|\psi(k')|>N^{2/3}\tau_2 z^{-1}\}}.
\end{split}\end{equation}
%%%%%%%
%
By (\ref{in:prob1}), (\ref{in:prob1b})  we find that
for $n,m\geq 2$
\begin{equation}\label{in:prob2}\begin{split}
&\int_0^\infty du\;e^{-u}\int_\T dk_1\;
p^{n}(k,k_1)\1_{\{|\psi(k_1)|>N^{2/3}\tau_1 u^{-1}\}}
\int_0^\infty dz\; e^{-z}\\
& \quad \times \int_\T dk'\;
p^{m}(k_1,k')\1_{\{|\psi(k')|>N^{2/3}\tau_2 z^{-1}\}}\\
&\leq C_2
\left[
\frac{\cos^2(\pi k)}{1+2\cos^2(\pi k)}
+
\frac{\sin^2(\pi k)}{1+2\cos^2(\pi k)}
\right]
\frac{1}{N^2}\frac{1}{\tau_1^{3/2}}\frac{1}{\tau_2^{3/2}},
\end{split}\end{equation}
with $C_2>0$, while if $m=1$ and/or $n=1$, then (\ref{prob2}) is of
order $N^{-8/3}$.
By repeating this procedure, one finds that for
every $n_1,..,n_\ell\geq 2$ the following inequality holds:
\begin{equation}\label{in:prob}\begin{split}
%\lim_{N\to\infty}N^\ell
& \int_0^\infty dz_1\; e^{-z_1}
\int_\T dk_1\;
p^{n_1}(k,k_1)\1_{\{|\psi(k_1)|>N^{2/3}\tau_1 z_1^{-1}\}}
\int_0^\infty dz_2\; e^{-z_2}\\
& \quad \times
\int_\T
dk_2\;p^{n_2}(k_1,k_2)\1_{\{|\psi(k_2)|>N^{2/3}\tau_2
z_2^{-1}\}}
\int_0^\infty dz_3 e^{-z_3}\int_\T....
\int_0^\infty dz_\ell e^{-z_\ell}\\
& \quad \times
\int_\T
dk_\ell\;p^{n_\ell}(k_{\ell-1},k_\ell)\1_{\{|\psi(k_\ell)|>N^{2/3}\tau_\ell
z_\ell^{-1}\}}\\
&\leq
\frac{1}{N^\ell}C_\ell\frac{1}{\tau_1^{3/2}}...\frac{1}{\tau_\ell^{3/2}}\left[
\frac{\cos^2(\pi k)}{1+2\cos^2(\pi k)}
+
\frac{\sin^2(\pi k)}{1+2\cos^2(\pi k)}
\right].
\end{split}\end{equation}
where $C_\ell$ are finite constants for $n_1,..,n_\ell\geq 2$.  If
$n_i=1$ for some $i$, then the l. h. s. of (\ref{in:prob}) is of order
$o(N^{-\ell})$.

Assume without loss of generality,  $i_1<i_2<..<i_\ell$.h
Then
\begin{equation}\begin{split}
& N^\ell \; \P\left[e_1 |\psi_{i_1}|>N^{2/3}\tau_1,...,
  e_l|\psi_{i_\ell}|>N^{2/3}\tau_\ell\right]\\
&= N^\ell
\int_\T d\mu(k_0)
\int_0^\infty dz_1\; e^{-z_1}
\int_\T dk_1\;
p^{i_1}(k_0,k_1)\1_{\{|\psi(k_1)|>N^{2/3}\tau_1 z_1^{-1}\}}\\
& \quad \times
\int_0^\infty dz_2\; e^{-z_2}
\int_\T
dk_2\;p^{i_2-i_1}(k_1,k_2)\1_{\{|\psi(k_2)|>N^{2/3}\tau_2
z_2^{-1}\}}
\int_0^\infty dz_3 e^{-z_3}\\
&\quad\times\int_\T...
\int_0^\infty dz_\ell e^{-z_\ell}
\int_\T
dk_\ell\;p^{i_\ell-i_{\ell-1}}(k_{\ell-1},k_\ell)\1_{\{|\psi(k_\ell)|>N^{2/3}\tau_\ell
z_\ell^{-1}\}}.
\end{split}\end{equation}
The proof of the lemma is now just an application of (\ref{in:prob}).
\end{proof}

%%%%%%%%%%%%%%%%%%%%%%%%%%%%%%%%%
%%%%%%%%%%%%%%%%%%%%%%%
Recalling the definition of $X_{i,N}^+$, $X_{i,N}^-$ given in
(\ref{def:X+X-}),
%We observe that for every
%$(i_1,..,i_\ell)\in\{1,..,N\}$, $1\leq\ell\leq N$,
by the symmetry of the
probability density we have
\begin{equation}\begin{split}
&\P\left[\big|X_{i_1,N}^+\big|>\tau_1,...,
\big|X_{i_\ell,N}^+\big|>\tau_\ell\right]=
\P\left[\big|X_{i_1,N}^-\big|>\tau_1,...,
\big|X_{i_\ell,N}^-\big|>\tau_\ell\right]\\
&=\P\left[e_1 \big|\psi_{i_1}\big|>N^{2/3}\tau_1,...,
e_l\big|\psi_{i_\ell}\big|>N^{2/3}\tau_\ell\right].
\end{split}\end{equation}
%for all sets of positive constants $(\tau_1,..,\tau_\ell)$. Then, in
%order to verify the conditions of Theorem \ref{theorem1}, it is enough
%to show that for every $\ell\geq 1$ and all sets of positive constants
%$(\tau_1,..,\tau_\ell)$
%\begin{equation}\label{prop8}
%\sum_{\alpha_N(\ell)}\P\left[e_1 \big|\psi_{i_1}\big|>N^{2/3}\tau_1,...,
%e_l\big|\psi_{i_\ell}\big|>N^{2/3}\tau_\ell\right]\to
%a_0^\ell\frac{1}{\tau_1^{3/2}}...\frac{1}{\tau_\ell^{3/2}},
%\hspace{0.4cm}N\to\infty,
%\end{equation}
%with some $a_0>0$.
Let us denote by $\sum_{\beta_N(\ell)}$ the sum over all not ordered
sequences of different indices $[i_1,..,i_\ell]$ with values in
$\{1,..,N\}$.  Then $\sum_{\alpha_N(\ell)}=\ell!\sum_{\beta_N(\ell)}$.

We choose $i_1< i_2 <...<i_\ell$ and by denoting with $m_1=i_1$,
$m_j=i_j-i_{j-1}$, $\forall j=2,..,\ell$, we get
\begin{equation}\begin{split}
& \sum_{\alpha_N(\ell)}\P\left[e_1 \big|\psi_{i_1}\big|>N^{2/3}\tau_1,...,
e_l\big|\psi_{i_\ell}\big|>N^{2/3}\tau_\ell\right]\\
 & =
\ell!\sum^{N}_{
 {m_1,..,m_\ell=1}\atop
 {m_1+..+m_\ell\leq N}}
\P\left[e_1 \big|\psi_{m_1}|>N^{2/3}\tau_1,...,
e_l\big|\psi_{(m_1+...+m_\ell)}\big|>N^{2/3}\tau_\ell\right].
\end{split}\end{equation}
Fixed $M\ll N$, we split the sum on $m_1$ into two parts:
\begin{equation}\label{sum}\begin{split}
& \sum^{N}_{
 {m_1,..,m_\ell=1}\atop
 {m_1+..+m_\ell\leq N}}
\P\left[e_1 \big|\psi_{m_1}|>N^{2/3}\tau_1,...,
e_l\big|\psi_{(m_1+...+m_\ell)}\big|>N^{2/3}\tau_\ell\right]\\
& =
\sum_{m_1=1}^M
\sum^{N}_{
{m_2,..,m_\ell=1}\atop
 {m_1+..+m_\ell\leq N}}
\P\left[e_1 \big|\psi_{m_1}\big|>N^{2/3}\tau_1,...,
e_l\big|\psi_{(m_1+...+m_\ell)}\big|>N^{2/3}\tau_\ell\right]\\
& \quad +
\sum_{m_1=M+1}^{N-\ell+1}
\sum^{N}_{
{m_2,..,m_\ell=1}\atop
{m_1+..+m_\ell\leq N}}
\P\left[e_1 \big|\psi_{m_1}\big|>N^{2/3}\tau_1,...,
e_l\big|\psi_{(m_1+...+m_\ell)}\big|>N^{2/3}\tau_\ell\right]
\end{split}\end{equation}
%Let us consider the first sum.
Using Lemma \ref{lemma1},
we find that the first sum on the r. h. s. of (\ref{sum}) is bounded by
$\tilde{C}_\ell \tau_1^{-3/2}...\tau_\ell^{-3/2}M/N$,
with $\tilde{C}_\ell>0$.
Let us consider the second sum on the r. h. s. of (\ref{sum}).
We split the sum on $m_2$ into two
parts and we get
\begin{eqnarray}
& & \sum_{m_1=M+1}^{N-\ell+1}
 \sum^{N}_{
{m_2,..,m_\ell=1}\atop {m_1+..+m_\ell\leq N}}
\P\left[e_1 \big|\psi_{m_1}\big|>N^{2/3}\tau_1,...,
e_l\big|\psi_{(m_1+...+m_\ell)}\big|>N^{2/3}\tau_\ell\right]\\
& & \nonumber =
 \sum_{m_1=M+1}^{N-\ell+1}
 \sum_{m_2=M+1}^{N-\ell+1}
 \sum^{N}_{
 {m_3,..,m_\ell=1}\atop{m_1+..+m_\ell\leq N}}
\P\left[e_1 \big|\psi_{m_1}\big|>N^{2/3}\tau_1,...,
e_l\big|\psi_{(m_1+...+m_\ell)}\big|>N^{2/3}\tau_\ell\right]\\
& & \nonumber \quad+
\mathcal{O}(M/N).
\end{eqnarray}

By repeating this procedure for all the sums, finally we get
\begin{eqnarray}\label{eq1}
  & & \sum^{N}_{{m_1,..,m_\ell=1}\atop
    {m_1+..+m_\ell\leq N}}
  \P\left[e_1 \big|\psi_{m_1}|>N^{2/3}\tau_1,...,
  e_l\big|\psi_{(m_1+...+m_\ell)}\big|>N^{2/3}\tau_\ell\right]\\
  & &\nonumber=
  \sum^{N}_{{m_1,..,m_\ell=M+1}\atop
    {m_1+..+m_\ell\leq N}}
  \P\left[e_1 \big|\psi_{m_1}\big|>N^{2/3}\tau_1,...,
  e_l\big|\psi_{(m_1+...+m_\ell)}\big|>N^{2/3}\tau_\ell\right]+\mathcal{O}(M/N).
\end{eqnarray}

Now we show that in the last expression we can replace the probability $\P$ with the
invariant measure $\pi$. We have
\begin{equation}\label{eq2}\begin{split}
& \sum^{N}_{{m_1,..,m_\ell=M+1}\atop
  {m_1+..+m_\ell\leq N}}
\P\left[e_1 \big|\psi_{m_1}\big|>N^{2/3}\tau_1,...,
e_l\big|\psi_{(m_1+...+m_\ell)}\big|>N^{2/3}\tau_\ell\right]\\
& =
\sum^{N}_{{m_1,..,m_\ell=M+1}\atop
{m_1+..+m_\ell\leq N}}
\int_\T d\mu(k_0)\int_0^\infty dz_1
\int_\T dk_1\; p^{m_1}(k_0, k_1)\1_{\{|\psi(k_1)|>
N^{2/3}\tau_1 z_1^{-1}\}}\\&\quad\times\int_\T...
\int_0^\infty dz_\ell\int_\T dk_\ell\;\phi(k_\ell)\1_{\{|\psi(k_\ell)|>
N^{2/3}\tau_\ell z_\ell^{-1}\}}\\
 & \quad+
\sum^{N}_{{m_1,..,m_\ell=M+1}\atop
{m_1+..+m_\ell\leq N}}
\int_\T d\mu(k_0)\int_0^\infty dz_1\int_\T dk_1\; p^{m_1}(k_0, k_1)\1_{\{|\psi(k_1)|>
  N^{2/3}\tau_1  z_1^{-1}\}}\\&\quad\times\int_\T...
\int_0^\infty dz_\ell\int_\T  dk_\ell\; [p^{m_\ell}(k_{\ell-1},
k_\ell)-\phi(k_\ell)]
\1_{\{|\psi(k_\ell)|>
  N^{2/3}\tau_\ell z_\ell^{-1}\}}.
\end{split}\end{equation}
Let us consider the second sum on the r.h.s. of (\ref{eq2}). We have
\begin{eqnarray}
&&    \left|\int_0^\infty dz_\ell\int_\T dk_\ell\;
      \left[p^{m_\ell}(k_{\ell-1}, k_\ell)-\phi(k_\ell)\right]
      \1_{\{|\psi(k_\ell)|>
        N^{2/3}\tau_\ell z_\ell^{-1}\}}\right| \\
&&\nonumber \leq   \sup_{k,k'\in\T}\left\{\phi(k')^{-1}\big|p^{m_\ell}(k,k')-\phi(k')\big|\right\}
    \int_0^\infty dz_\ell\int_\T dk_\ell\; \phi(k_\ell)
    \1_{\{|\psi(k_\ell)|> N^{2/3}\tau_\ell z_\ell^{-1}\}},
\end{eqnarray}
which is finite since $p^m(k,k')\phi(k')^{-1}\in
\mathbb{C}^\infty(\T\times\T)$ for all $m\geq 1$. Moreover, since
by direct computation
$$
\int_0^\infty dz\;\int_\T dk\;
\phi(k)
\1_{\{|\psi(k)|>
N^{2/3}\tau z^{-1}\}}\leq C_0\frac{1}{N}\frac{1}{\tau^{3/2}},
$$
with $C_0>0$, using Lemma \ref{lemma1}, we find that the second sum
in (\ref{eq2}) is bounded, for  $N$ large enough, by
\begin{equation}\begin{split}
&\frac{(N-\ell-M)^\ell}{N^\ell}C_{\ell}\frac{1}{\tau_1^{3/2}}...
    \frac{1}{\tau_\ell^{3/2}}
\sup_{m\geq M+1}\sup_{k,k'\in\T}
    \left\{\phi(k')^{-1} \big|p^{m}(k,k')-\phi(k')\big|\right\}\\
& \leq
    \tilde{C}_\ell\frac{1}{\tau_1^{3/2}}...  \frac{1}{\tau_\ell^{3/2}}
 \sup_{m\geq  M+1}
\sup_{k,k'\in\T}\left\{\phi(k')^{-1}\big|p^{m}(k,k')-\phi(k')\big|\right\}.
\end{split}\end{equation}
Now we consider the first sum of (\ref{eq2}).
By repeating this procedure $\ell$-times,
we can replace the transition probability
densities $p^m_i(k,k')$, $i=1,..,\ell-1$,
with the invariant density $\phi(k')$.
This gives an error $\mathcal{E}_N(\ell,M,\tau_1,..,\tau_\ell)$, which
satisfies the inequality
\begin{equation}\label{E}
\big|\mathcal{E}_N(\ell,M,\tau_1,..,\tau_\ell)\big|\leq C_\ell\frac{1}{\tau_1^{3/2}}...
\frac{1}{\tau_\ell^{3/2}}
\sup_{m\geq
  M+1}\sup_{k,k'\in\T}\left\{\phi(k')^{-1}\big|p^{m}(k,k')-\phi(k')\big|\right\}.
\end{equation}
By ergodicity, for every $\ell\geq 1$ and
 all sets of constants $\tau_i>0$, $i=1,..,\ell$
\begin{equation}\label{lim_E}
\lim_{M\to\infty}\big|\mathcal{E}_N(\ell,M,\tau_1,..,\tau_{\ell})\big|=0,
\quad\forall N\in\mathbb{N}.
\end{equation}

Then, for  $N$
large enough,
\begin{eqnarray}
&& \sum^{N}_{{m_1,..,m_\ell=1}\atop
  {m_1+..+m_\ell\leq N}}
\P\left[e_1 \big|\psi_{m_1}\big|>N^{2/3}\tau_1,...,
  e_l\big|\psi_{(m_1+...+m_\ell)}\big|>N^{2/3}\tau_\ell\right]\\
&&\nonumber  =\sum^{N}_{{m_1,..,m_\ell=M+1}\atop
   {m_1+..+m_\ell\leq N}}
 \Pi_{i=1}^{\ell}
\left(
\int_0^\infty dz\;\int_\T dk\; \phi(k)\1_{\{|\psi(k)|>
N^{2/3}\tau_i z^{-1}\}}
\right)\\
&&\nonumber \quad +
\mathcal{E}_N(\ell,M,\tau_1,..,\tau_\ell)
+\mathcal{O}(M/N)\\
&& \nonumber =
\sum^{N}_{{m_1,..,m_\ell=M+1}\atop
 {m_1+..+m_\ell\leq N}}
 \left[
\frac{a^\ell}{N^\ell}\frac{1}{\tau_1^{3/2}}...\frac{1}{\tau_\ell^{3/2}}+\mathcal{O}(N^{-5/3})
\right]+\mathcal{E}_N(\ell,M,\tau_1,..,\tau_\ell)
+\mathcal{O}(M/N).
\end{eqnarray}
We take the limit $N,M\to \infty$ with $M/N\to 0$. The proof of the
proposition follows by (\ref{lim_E})
and by relation
$$
\lim_{N\to\infty}\sum^{N}_{
{m_1,..,m_\ell=1}\atop
 {m_1+..+m_\ell\leq N}}
N^{-\ell}=1/\ell!
$$

%%%%%%%%%%%%%%%%%%%%%%%%%%%%%%%%%%%%%%
%%%%%%%%%%%%%%%%%%%%%%%%%%%%%%%%%%%%%%%%%
\section{Tightness.} In this section we conclude the proof of Theorem
\ref{theo:conv}
 by proving
 tightness of the sequence $\st$. This is relatively easy due to the
strong convergence properties stemming from the weak convergence of the
 point processes $\cR_N$.

As criterion for tightness
in the $J_1$ topology we use  slight variant of  Theorem 13.2 (with
condition 13.5 replaced by 13.8) from \cite{Bi}.

We define  the modulo of continuity
on $D$
\begin{equation}\begin{split}
%&
w_f(\delta)=\sup\{(|f(t)-f(t_1)|\wedge|f(t_2)-f(t)|): t_1\leq
t\leq t_2, t_2-t_1\leq \delta\}.
%\\ &v_f(\delta)=\sup\{|f(t_1)-f(t_2)|:t_1,t_2\in[0,T]\cap (t-\delta, t+\delta)\}.
\end{split}\end{equation}
The sequence $\{P_n\}$ of probability measures on $(D, \mathcal{D})$
is tight in the $J_1$-topology if and only if
\begin{itemize}
\item[(i)]
For each positive $\epsilon$ there exist $\tau$ such that
\begin{equation}\label{i}
P_n[f:\|f\|_\infty> \tau]\leq \epsilon,\hspace{1cm}n\geq 1.
\end{equation}
\item[(ii)]For each $\epsilon >0$ and $\eta >0$, there exist $\delta>0$,
and a integer $n_0$ such that $\forall n\geq n_0$
\begin{equation}
P_n[f:w_f(\delta)\geq \eta]\leq \epsilon,\hspace{1cm}n\geq n_0,
\end{equation}
and
\begin{eqnarray}
& &\label{ii1}
P_n[f:w_f(\delta)\geq \eta]\leq \epsilon,\\
 & &\label{ii2}
P_n[f:|f(\delta)-f(0)|\geq \eta]\leq \epsilon,\\
& &\label{ii3}
P_n[f:|f(\mathcal{T})-f(\mathcal{T}-\delta)|\geq \eta]\leq \epsilon.
\end{eqnarray}
\end{itemize}

In order to verify these conditions,
we start with some definitions and preliminary results.
For every $0\leq r\leq s\leq t\leq \mathcal{T}$, let us denote by
\begin{equation}\label{mrst}
m_{rst}=  \min\left(\big|\st(s)-\st(r)\big|,\big|\st(t)-\st(s)\big|\right)
\end{equation}
and
\begin{equation}\label{L_T}
L_\mathcal{T}= \sup_{0\leq r\leq s\leq t\leq \mathcal{T}}m_{rst}
\end{equation}
The following inequalities holds (see \cite{Bi}, (10.4), (10.6)):
\begin{eqnarray}
&& \label{ineq1}
\sup_{\theta\in[0,\mathcal{T}]}|\st(\theta)| \leq L_\mathcal{T}
+ |\st(\mathcal{T})|,\\
&&\label{ineq2}
\sup_{\theta\in[0,\mathcal{T}]}|\st(\theta)|\leq 3 L_\mathcal{T} +
\max_{0\leq i\leq [N\mathcal{T}]-1}N^{-2/3}e_i|\psi_i|.
\end{eqnarray}

Fix $N\geq 1$. We have
\begin{equation}\label{ineq3}
\P\left[ \sup_{\theta\in[0,\mathcal{T}]}\st(\theta)>\tau\right]\leq
\P\left[ \sup_{\theta\in[0,\mathcal{T}]}\sti^<(\theta)>\tau/2\right]+
\P\left[ \sup_{\theta\in[0,\mathcal{T}]}\sti^>(\theta)>\tau/2\right],
\end{equation}
where
\begin{equation}\label{def:sti}\begin{split}
\sti^<(\theta) &
=\frac{1}{N^{2/3}}\sum_{i=0}^{[N\theta]-1}e_i\psi_i\1_{\{e_i|\psi_i|<N^{2/3}\}}
\\
\sti^>(\theta) &
=\frac{1}{N^{2/3}}\sum_{i=0}^{[N\theta]-1}e_i\psi_i\1_{\{e_i|\psi_i|>N^{2/3}\}}.
\end{split}\end{equation}
% $\st^>, \st^<$  defined in (\ref{def:S+S-}).
%\emph{In the proof of (i) we do not need to send $c\to 0$, so we
  %consider $c$ fixed.}
Let us consider inequality (\ref{ineq1}) with $\sti$ replaced by
$\sti^<$.  We have
\begin{equation}\label{ineq4}
\P\left[ \sup_{\theta\in[0,T]}\sti^<(\theta)>\tau/2\right]\leq
\P\left[ L^<_T> \tau/4\right]+\P\left[\big|\sti^<(T)\big|>\tau/4\right].
\end{equation}
%where
%\begin{equation}\begin{split}
%L^<_T= &\sup_{0\leq r\leq s\leq t\leq T}m^<_{rst},\\
%m^<_{rst}= & \min\left(|\st^<(s)-\st^<(r)|,|\st^<(t)-\st^<(s)|\right).
%\end{split}\end{equation}
%
By direct computation, for all $0\leq r\leq s\leq t\leq \mathcal{T}$ we have
by Lemma \ref{ESS}
\begin{equation}\label{ineqm<}\begin{split}
\P\left[m^<_{rst}\geq \tau/4\right]\leq
\P\left[\big|\sti^<(s)-\sti^<(r)\big|\big|\sti^<(t)-\sti^<(s)\big|
\geq \tau^2/16\right]\\
\leq \left(\frac{16}{\tau^2}\right)^{2}
\E\left[\big|\sti^<(s)-\sti^<(r)\big|^2
    \big|\sti^<(t)-\sti^<(s)\big|^2\right]\leq \frac{A}{\tau^4}(t-r)^2.
\end{split}\end{equation}
Then we can use Theorem 10.3 of \cite{Bi} to get
\begin{equation}\label{ineqL<}
\P\left[L^<_T\geq \tau/4\right]\leq \frac{B}{\tau^4} \mathcal{T}^2,
\end{equation}
where $B$ is a  constant. Moreover, by inequality (\ref{S-})
with $c=1$, $\theta=\mathcal{T}$ we get
\begin{equation}\label{ineq5}
  \P\left[\big|\sti^<(\mathcal{T})\big|\geq \tau/4\right]
  \leq
  \frac{16}{\tau^2}\E\left[\big|\sti^<(\mathcal{T})\big|^2\right]
  \leq \frac{D}{\tau^2}\mathcal{T},
\end{equation}
with some constant $D$.  Finally by (\ref{ineq4}), (\ref{ineqL<})
and (\ref{ineq5}), we get that for every $N\geq 1$
\begin{equation}\label{boundsti<}
  \P\left[\sup_{\theta\in[0,\mathcal{T}]}\sti^<(\theta)\geq \tau/2\right]\leq
  \frac{B}{\tau^4} \mathcal{T}^2
  + \frac{D}{\tau^2}\mathcal{T}.
\end{equation}

In order to estimate
$\P\left[\sup_{\theta\in[0,\mathcal{T}]}\sti^>(\theta)> \tau/2\right]$, we observe
that
$$
\sup_{\theta\in[0,\mathcal{T}]}\sti^>(\theta)\leq \frac{1}{N^{2/3}}
\sum_{i=0}^{[N\mathcal{T}]-1}e_i|\psi_i|\1_{\{e_i|\psi_i|>N^{2/3}\}}.
$$
Then by a straightforward computation (see Lemma \ref{mom1})
\begin{equation}\begin{split}
&\P\left[\sup_{\theta\in[0,\mathcal{T}]}\sti^>(\theta)>
      \tau/2\right]
\leq \frac{2}{\tau}\frac{1}{N^{2/3}}
    \sum_{i=0}^{[N\mathcal{T}]-1}\E\left[e_i|\psi_i|\1_{\{e_i|\psi_i|>N^{2/3}\}}\right]\\
 & \leq
    \frac{2}{\tau}\frac{1}{N^{2/3}}N\mathcal{T}\frac{K_0}{N^{1/3}}=
    2K_0\mathcal{T}\frac{1}{\tau},
\end{split}\end{equation}
with constants $K_0$ . This inequality with (\ref{ineq3}) and
(\ref{boundsti<}) prove (\ref{i}).

Now we prove (\ref{ii1}). Again, for each $\eta>0$
\begin{equation}\label{ineq9}
\P\left[ w_{\sti}(\delta)\geq \eta\right]\leq
\P\left[ w_{\sti^<}(\delta)\geq \eta/2\right] +
\P\left[ w_{\sti^>}(\delta)\geq \eta/2\right].
\end{equation}
To estimate $\P\left[ w_{\sti^<}(\delta) \geq \eta/2\right]$, we divide
$[0,\mathcal{T}]$ in  intervals of length $\delta>0$.
If $ w_{\sti^<}(\delta)\geq \eta/2$, then exists $k$,
$1\leq k <\mathcal{T}/\delta$, such that
\begin{equation*}\begin{split}
   & \sup\left\{
      \min\left(\big|\sti^<(s)-\sti^<(r)\big|,\big|\sti^<(t)-\sti^<(s)\big|
\right),\;(k-1)\delta\leq r\leq s\leq t\leq
      (k+1)\delta\right\}\\
   & \geq\eta/2.
\end{split}\end{equation*}
Observe that
$$
\min\left(\big|\sti^<(s)-\sti^<(r)\big|,\big|\sti^<(t)-\sti^<(s)\big|\right)=m^<_{rst}.
$$
Thus, by (\ref{ineqm<}),
\begin{equation}
\P\left[m^<_{rst}\geq \eta/2 \right]\leq A\frac{1}{\eta^4}(t-r)^2,
\hspace{1cm}(k-1)\delta\leq r\leq t\leq (k+1)\delta.
\end{equation}
Hence
\begin{equation}
\P\left[\sup\left\{m^<_{rst}: (k-1)\delta\leq r\leq s\leq t\leq
(k+1)\delta\right\} \geq \eta/2\right]\leq \frac{C_0}{\eta^4}\delta^2,
\end{equation}
with some constant $C_0$  (see  Theorem 10.3 in \cite{Bi}).
Using this result we get
\begin{equation}\label{ineq10}\begin{split}
&\P\left[ w_{\sti^<}(\delta)\geq \eta/2\right]\\
&\leq
\sum_{k=1}^{[\mathcal{T}/\delta]}\P\left[\sup\left\{m^<_{rst}: (k-1)\delta\leq r\leq s\leq t\leq
(k+1)\delta\right\}\geq \eta/2\right]\\
&\leq \frac{C_1}{\eta^4}\mathcal{T}\delta
\end{split}\end{equation}
with some constant $C_1$.

Now let us consider the quantity
$$m^>_{rst}=\min\left(\big|\sti^>(s)-\sti^>(r)\big|,\big|\sti^>(t)-\sti^>(s)\big|\right)$$
for $0\leq r\leq s\leq t\leq \mathcal{T}$, with $t-r<\delta$.  If
$m^>_{rst}>0$, then $e_i|\psi_i|>N^{2/3}$, $e_j|\psi_j|>N^{2/3}$ for
some $[r N]\leq i\leq [sN]-1$, $[s N]\leq j\leq [tN]-1$.
We get
\begin{equation}\label{ineq11}\begin{split}
&\P\left[ w_{\sti^>}(\delta) \geq \eta/2\right]\\
& \leq\P\left[\exists\;  i,j< N\mathcal{T}, |i-j|\leq 2\delta N:
e_i|\psi_i|>N^{2/3}, \;e_j|\psi_j|>N^{2/3} \right]\\
&\leq \sum_{i=0}^{[N\mathcal{T}]-1}\sum_{j=i+1}^{i+[2 N \delta]}
\P\left[e_i|\psi_i|>N^{2/3}, \;e_j|\psi_j|>N^{2/3} \right]\\
&\leq C\mathcal{T}\delta
\end{split}\end{equation}
 where  the last inequality follows by Lemma \ref{lemma1}.
Equation (\ref{ineq11}) together with (\ref{ineq10}) proves
(\ref{ii1}).

The proofs of equations (\ref{ii2}),(\ref{ii3}) are easier. We give only the
proof of (\ref{ii3}), since the other is similar. We have
\begin{equation}\begin{split}
&\P\left[\big|S_N(\mathcal{T})-S_N(\mathcal{T}-\delta)\big|\geq \eta\right]
 \\ & \leq
\P\left[\big|\sti^<(\mathcal{T})-\sti^<(\mathcal{T}-\delta)\big|\geq \eta/2\right]
+
\P\left[\big|\sti^>(\mathcal{T})-\sti^>(\mathcal{T}-\delta)\big|\geq \eta/2\right],
\end{split}\end{equation}
where, using a second moment estimate (essentially the same proof of
Lemma \ref{proof:S-}), we get
\begin{equation}\label{ineq12}\begin{split}
&\P\left[\big|\sti^<(\mathcal{T})-\sti^<(\mathcal{T}-\delta)\big|\geq
    \eta/2\right]\\
& = \P\left[ \left|\frac{1}{N^{2/3}}\sum_{n=\lfloor
          N(\mathcal{T}-\delta)\rfloor}^{\lfloor N\mathcal{T}\rfloor
          -1}e_n\psi_n\1_{\{|e_n\psi_n|\leq N^{2/3}\}}\right| \geq
      \eta/2
    \right]\\
& \leq \frac{4}{\eta^2}\frac{1}{N^{4/3}}
\E\left[
      {\left|\sum_{n=\lfloor N(\mathcal{T}-\delta)\rfloor}^{\lfloor
            N\mathcal{T}\rfloor -1}e_n\psi_n\1_{\{|e_n\psi_n|
            N^{2/3}\}} \right|}^2 \right]\\
&\leq\frac{C}{\eta^2}\delta.
\end{split}\end{equation}
 Moreover,
\begin{equation}\begin{split}
&\P\left[\big|\sti^>(\mathcal{T})-\sti^>(\mathcal{T}-\delta)\big|
   \geq \eta/2\right]\\
& \leq
 \P\left[ \frac{1}{N^{2/3}}\sum_{n=\lfloor N(\mathcal{T}-\delta)\rfloor}^{\lfloor
  N\mathcal{T}\rfloor -1}|e_n\psi_n|\;\1_{\{|e_n\psi_n|> N^{2/3}\}}
\geq \eta/2
\right]\\
& \leq \frac{2}{\eta}\frac{1}{N^{2/3}}\sum_{n=\lfloor N(\mathcal{T}-\delta)\rfloor}^{\lfloor
  N\mathcal{T}\rfloor -1}
\E\left[|e_n\psi_n|\;\1_{\{|e_n\psi_n|> N^{2/3}\}}
\right]\\
&\leq \frac{K}{\eta}\delta,
\end{split}\end{equation}
and this together with (\ref{ineq12}) proves (\ref{ii2}).

\section{Proof Theorem \ref{theo:conv2}}\label{proof:
conv2}
Let us consider the two rescaled processes
$T_N$, $T_N^{-1}$ defined in (\ref{def:TN.SN}).
We want to prove the convergence in probability of both  processes
to the function $\theta$.
We start with the following Lemma.
\begin{lemma}\label{conv0:TN}
Let $T_N$ be the process defined in (\ref{def:TN.SN}). Then
$\forall \varepsilon>0$
\begin{equation}
\lim_{N\to\infty} \P\left[\big|T_N(\theta)-\theta\big|>\varepsilon\right]=0,
\hspace{1cm}\forall \theta\in[0,\mathcal{T}].
\end{equation}
\end{lemma}

\begin{proof} Le us denote by $\E_\pi$ the expectation value with
respect to the invariant measure $\pi$. Since
$\E_\pi\left[e_i\phi(X_i)^{-1}\right]=1$, $\forall i=0,1,...$, then
\begin{equation*}
\P\left[\big|T_N(\theta)-\theta\big|>\varepsilon\right]\leq
\P\left[\frac{1}{N} \Big|\sum_{i=0}^{\scriptstyle{\lfloor N\theta\rfloor -1}}
\left(e_i\phi(X_i)^{-1}-\E_\pi\left[e_i\phi(X_i)^{-1}\right]\right)\Big|
>\varepsilon-\frac{1}{N} \right].
\end{equation*}
We denote by $\varepsilon'=\varepsilon-N^{-1}$, which is positive
for $N$ large enough.
Let us introduce the following notations:
\begin{equation}\label{def:tau}
\tau_i^< \equiv e_i\phi(X_i)^{-1}\1_{\{e_i\phi(X_i)^{-1}\leq N^{2/3}\}},
\hspace{0.4cm}
\tau_i^> \equiv e_i\phi(X_i)^{-1}\1_{\{e_i\phi(X_i)^{-1}>N^{2/3}\}},
\end{equation}
$\forall i=0,1,...$.
We have
\begin{equation}\label{T1}\begin{split}
& \P\left[\frac{1}{N} \Big|\sum_{i=0}^{\scriptstyle{\lfloor N\theta\rfloor -1}}
\left(e_i\phi(X_i)^{-1}-\E_\pi\left[e_i\phi(X_i)^{-1}\right]\right)\Big|
>\varepsilon'\right]\\
& \leq
\P\left[\frac{1}{N}
\Big|\sum_{i=0}^{\scriptstyle{\lfloor N\theta\rfloor -1}}
\left(\tau_i^< -1\right)\Big|
>\frac{\varepsilon'}{2}\right]
+
\P\left[\frac{1}{N}
   \sum_{i=0}^{\scriptstyle{\lfloor N\theta\rfloor -1}}
   \tau_i^>
   >\frac{\varepsilon'}{2}\right].
\end{split}\end{equation}
By a first moment estimation (see Lemma \ref{mom1}), we get
\begin{equation*}
\E\left[N^{-1}
      \sum_{i=0}^{\scriptscriptstyle{\lfloor N\theta\rfloor -1}}
        \tau_i^>
   \right]\leq A_0\mathcal{T}\frac{1}{N^{1/3}},
\end{equation*}
with $A_0 <\infty$, thus we can neglect the second term on the r.h.s. of
(\ref{T1}).
For the first term we have
\begin{equation}\label{T2}\begin{split}
&\P\left[\frac{1}{N}
\Big|\sum_{i=0}^{\scriptstyle{\lfloor N\theta\rfloor -1}}
\left(\tau_i^< -1\right)\Big|
>\frac{\varepsilon'}{2}\right]\\
&\leq
\frac{4}{\varepsilon^2}\frac{1}{N^2}
    \sum_{i=0}^{\scriptstyle{\lfloor N\theta\rfloor -1}}
    \E \left[\left(\tau_i^< -1 \right)^2\right]
+    \frac{4}{\varepsilon^2}\frac{1}{N^2}
\sum_{\scriptstyle{i,j=0,\; i\neq j}}^{\scriptstyle{\lfloor
    N\theta\rfloor -1}}
    \E \left[\tau_i^<  \tau_j^< - \E_\pi\left[\tau_i^<\right]\E_\pi \left[\tau_j^<\right]
\right]\\
& \quad +\mathcal{O}(N^{-1/3}),
\end{split}\end{equation}
where, using a second moment estimate
(see for example (\ref{ESpsi2})),
\begin{equation}\begin{split}
\frac{4}{\varepsilon^2}\frac{1}{N^2}
    \sum_{i=0}^{\scriptstyle{\lfloor N\theta\rfloor -1}}
    \E \left[\left(\tau_i^< -1 \right)^2\right]
   \leq \frac{1}{\varepsilon^2}B_0\mathcal{T}\frac{1}{N^{2/3}}.
   %+\mathcal{O}(N^{-1}).
\end{split}\end{equation}

%We rewrite  the second sum in (\ref{T2}) in the following way:
Let us consider the second sum in (\ref{T2}). For fixed $1<M<N$, we split
it into three parts:
\begin{equation}\label{T3}\begin{split}
& \sum_{i,j=0,\; i\neq j}^{\lfloor
    N\theta\rfloor -1}
    \E \left[\tau_i^<  \tau_j^< - \E_\pi\left[\tau_i^<\right]\E_\pi \left[\tau_j^<\right]
\right]\\
& =
%\frac{4}{\varepsilon^2}\frac{1}{N^2}
\sum_{{i,j=0}\atop{1\leq|j-i|\leq M}}
    ^{\lfloor N\theta\rfloor -1}
  \E \left[ \tau_i^<  \tau_j^< - \E_\pi\left[\tau_i^<\right]\E_\pi
  \left[\tau_j^<\right] \right]
+
%\frac{4}{\varepsilon^2}\frac{1}{N^2}
\sum_{i=0}^{M-1}
\sum_{{j=0}\atop{|j-i|>M}}
    ^{\lfloor N\theta\rfloor -1}
  \E \left[ \tau_i^<  \tau_j^< - \E_\pi\left[\tau_i^<\right]\E_\pi
    \left[\tau_j^<\right] \right]\\
& \quad+
%\frac{4}{\varepsilon^2}\frac{1}{N^2}
\sum_{i=M}^{\lfloor
    N\theta\rfloor -1}
\sum_{{j=0}\atop{|j-i|>M}}
    ^{\lfloor N\theta\rfloor -1}
\E \left[ \tau_i^<  \tau_j^< - \E_\pi\left[\tau_i^<\right]\E_\pi
    \left[\tau_j^<\right] \right].
\end{split}\end{equation}
For every $i,m\geq 1$,  we have
\begin{equation*}
%\label{Etau}
\begin{split}
\E\left[ \tau_i^<  \tau_{i+m}^<   \right]
=  & \int_0^\infty dz\; z e^{-z}
\int_0^\infty du\; u e^{-u}
\int_\T\;d\mu(k_0)\int_\T dk\;p^i(k_0,k)
\phi(k)^{-1}\1_{\{z \phi(k)^{-1}<N^{2/3}\}}\\
 & \times\int_\T dk'\; p^m(k,k')\phi(k')^{-1}\1_{\{z \phi(k')^{-1}<N^{2/3}\}}.
\end{split}\end{equation*}
Using (\ref{form:prob}) one finds that $\E \left[ \tau_i^<
\tau_{i+m}^<\right]\sim\mathcal{O}(1) $, thus the first and the second
sum in (\ref{T3}) are $\mathcal{O}(M/N)$. Moreover since
$\pi(dk)=\phi(k)dk$, one get
\begin{equation*}\begin{split}
& \E\left[ \tau_i^<  \tau_{i+m}^<   \right]\\
& =
\E_\pi\left[ \tau_i^<  \tau_{i+m}^<   \right]
+
\int_0^\infty dz\; e^{-z} z
\int_\T d\mu(k_0)
\int_\T dk\;p^i(k_0,k)
\phi(k)^{-1}\1_{\{z \phi(k)^{-1}<N^{2/3}\}}\\
& \quad\times
\int_0^\infty du\; e^{-u} u
\int_\T dk'\; \left(p^m(k,k')-\phi(k')\right)\phi(k')^{-1}
\1_{\{u \phi(k')^{-1}<N^{2/3}\}}\\
& \quad +
\int_0^\infty dz\;  e^{-z} z
\int_\T\;d\mu(k_0)
\int_\T dk\;\left(p^i(k_0,k)-\phi(k)\right)\phi(k)^{-1}
\1_{\{z \phi(k)^{-1}<N^{2/3}\}}\E_\pi\left[ \tau_{i+m}^< \right],
\end{split}\end{equation*}
and then $\forall i,m\geq 1$
\begin{equation*}
\Big|\E\left[ \tau_i^<  \tau_{i+m}^<   \right]-
\E_\pi\left[ \tau_i^< \right] \E_\pi\left[ \tau_{i+m}^<   \right]
\Big|
\leq C\int_\T dk \max_{i>M}\sup_{k_0\in\T}\left\{\phi(k)^{-1}
\big|  p^i(k_0,k)-\phi(k)  \big|\right\}.
\end{equation*}
Thus the third sum in (\ref{T3}) is bounded by
$$
\varepsilon^{-2}\mathcal{T}^2 C\int_\T dk \max_{i>M}\sup_{k_0\in\T}\left\{\phi(k)^{-1}
\big|  p^i(k_0,k)-\phi(k)  \big|\right\},
$$
which, by ergodicity, goes to zero for $M\to\infty$. We prove the
Lemma choosing $N, M\to \infty$, with $M/N\to 0$.
\end{proof}

%%%%%%%%%%%%%%%%%%%%%%%%%%%%%%%%%%%%%%%%%%%%%%%%%%%%%%
%%%%%%%%%%%%%%%%%%%%%%%%%%%%%%%%%%%%%%%%%%%%%%%%%

Now we prove that the processes $T_N(\theta)$,
$T_N^{-1}(\theta)$  converge in probability (and thus in
distribution) to the function $\theta$.
\begin{proposition}\label{conv:TN}
Let $T_N$, $T_N^{-1}$ be the processes defined in
(\ref{def:TN.SN}). Then $\forall \epsilon>0$, $\forall \delta>0$ $\exists
N_0$ such that $\forall N\geq N_0$
\begin{equation}\label{conv:TN.1}
\P\left[\sup_{\theta\in[0,\mathcal{T}]}\left\{\big|T_N(\theta)-\theta\big|\right\}>\epsilon
\right]<\delta,
\end{equation}
and
\begin{equation}\label{conv:TN-1}
\P\left[\sup_{\theta\in[0,\mathcal{T}]}\left\{\big|T_N^{-1}(\theta)-\theta\big|\right\}>\epsilon
\right]<\delta.
\end{equation}
\end{proposition}

\begin{proof} We give the proof only for
(\ref{conv:TN-1}), the proof of (\ref{conv:TN.1}) is similar.

Fixed $\eta>0$, we divide  $[0,\mathcal{T}]$ in
$\lfloor \mathcal{T}/\eta \rfloor +1$ intervals of length less or
equal than $\eta$. We observe that
\begin{equation}\begin{split}
\P\left[\sup_{\theta\in[0,\mathcal{T}]}\left\{\big|T_N^{-1}(\theta)-\theta
\big|\right\}>\epsilon\right] \leq
& \;\P\left[\sup_{\theta\in[0,\mathcal{T}]}\left\{T_N^{-1}(\theta)-\theta
  \right\}>\epsilon\right]\\
& +
\P\left[\sup_{\theta\in[0,\mathcal{T}]}\left\{\theta-T_N^{-1}(\theta)
\right\}>\epsilon\right].
\end{split}\end{equation}
We have
\begin{equation}\begin{split}
& \P\left[\sup_{\theta\in[0,\mathcal{T}]}\left\{T_N^{-1}(\theta)-\theta
\right\}>\epsilon\right] \leq \sum_{k=0}^{\lfloor\mathcal{T}/\eta\rfloor}
 \P\left[\sup_{\theta\in[k\eta, (k+1)\eta]}\left\{T_N^{-1}(\theta)-\theta
 \right\}>\epsilon\right]\\
& \leq \sum_{k=0}^{\lfloor \mathcal{T}/\eta\rfloor}\P\left[
 \left\{T_N^{-1}[(k+1)\eta]-(k+1)\eta
 \right\}>\epsilon-\eta\right]\\
&=
 \sum_{k=1}^{\lfloor \mathcal{T}/\eta\rfloor+1}\P\left[
 T_N^{-1}(k\eta)-k\eta
 >\epsilon-\eta\right]
\end{split}\end{equation}
where we used the fact that $T^{-1}_N$ is
non-decreasing.
For every $\epsilon>0$,
we choose $\eta$ such that $\epsilon'\equiv \epsilon-\eta$ is
positive, and we rewrite the relation $ T_N^{-1}(k\eta)-k\eta
 >\epsilon'$ as
\begin{equation}\label{ineq:T-1}
T^{-1}(N k\eta)>Nk\eta + N\epsilon',
\end{equation}
where $T^{-1}(t)$ is defined in (\ref{def:T-1}).
By definition  the following relation holds:
\begin{equation}
T_{[T^{-1}(N k\eta)-1]}<Nk\eta\leq
T_{[T^{-1}(Nk\eta)]},
\end{equation}
then inequality (\ref{ineq:T-1}) implies
\begin{equation}
Nk\eta>T_{\lfloor Nk\eta+N\epsilon'-1\rfloor},
\end{equation}
and thus
\begin{equation}\begin{split}
&\sum_{k=1}^{\lfloor \mathcal{T}/\eta\rfloor +1}\P\left[
 T_N^{-1}(k\eta)-k\eta
 >\epsilon'\right]\\
&\leq \sum_{k=1}^{\lfloor \mathcal{T}/\eta\rfloor +1}
 \P\left[ (k\eta +\epsilon'-N^{-1})-T_N(k\eta
 +\epsilon'- N^{-1})  > \epsilon'-N^{-1}
 \right].\end{split}
\end{equation}
In the same way one can easily prove  that
\begin{equation}
\P\left[\sup_{\theta\in[0,\mathcal{T}]}\left\{\theta-T_N^{-1}(\theta)
\right\}>\epsilon\right]
\leq
\sum_{k=0}^{\lfloor \mathcal{T}/\eta\rfloor}
  \P\left[ T_N(k\eta-\epsilon')-(k\eta-\epsilon')> \epsilon'\right].
\end{equation}
By Lemma \ref{conv0:TN}, $\forall \epsilon'>0$, $\forall \delta'>0$,
$\exists N_0$ such that $\P\left[ \big|T_N(\theta)-\theta\big| >
\epsilon'\right]<\delta'$, $\forall N\geq N_0$, for all
$\theta\in[0,\mathcal{T}]$, and this conclude the proof of
(\ref{conv:TN-1}).
\end{proof}

In order to prove that $S_N\circ T_N^{-1}$ converges in distribution
to the L\'evy process $V$,  we have just to show that the
distance between the two process $S_N$, $S_N\circ T_N^{-1}$ goes to
zero in probability (see Theorem 3.1 in \cite{Bi}). By
denoting with $\rho(\cdot,\cdot)$ the distance in the Skorokhod $J_1$
topology,
\begin{equation}\label{dist1}\begin{split}
\lim_{N\to\infty}\P\left[ \rho(S_N, S_N\circ T_N^{-1})>\epsilon
\right]
\leq\P\left[ \sup_{t\in[0,\mathcal{T}]}
\big|S_N(\lambda_N(t))-S_N\circ T_N^{-1}(t)\big|>\epsilon\right],
\end{split}\end{equation}
where $\lambda_N:[0,\mathcal{T}]\to [0,\mathcal{T}]$ is a sequence
of increasing homeomorphisms such that $\|\lambda_N -
I\|_{\infty}\to 0$. We have
\begin{equation*}\begin{split}
& \P\left[ \sup_{t\in[0,\mathcal{T}]} \big|S_N(\lambda_N(t))-S_N\circ
T_N^{-1}(t)\big|>\epsilon\right]\\
& \leq
  \P\left[ \left\{\sup_{t\in[0,\mathcal{T}]}
\big|S_N(\lambda_N(t))-S_N\circ T_N^{-1}(t)\big|>\epsilon\right\} \cap
\left\{\sup_{t\in[0,\mathcal{T}]}\big|T_N^{-1}(t)-t\big|\leq
\sigma\right\}\right]\\
& \quad + \P\left[
\sup_{t\in[0,\mathcal{T}]}\big|T_N^{-1}(t)-t\big|>\sigma\right],
\end{split}
\end{equation*}
for every $\sigma>0$. The second term on the r.h.s. goes to zero for
$N\to\infty$ (see Proposition \ref{conv:TN}), while the first term
is bounded by
$$
\P\left[\sup\left\{\big|S_N(\lambda_N(t))-S_N(s)\big|,
  |t-s|\leq\sigma,\; t,s\in[0,\mathcal{T}] \right\}>\epsilon\right],$$
which goes to zero for the tightness (see (\ref{ii1})) and the fact
that $\|\lambda_N - I\|_{\infty}\to 0$.

It remains to prove the convergence of the process
$Y_N$ to the stable process $V$. The basic idea is that the
step-function sequence $\{S_N(\theta)\}_{0\leq\theta\leq
  \mathcal{T}}$ and the continuous-time sequence, given by the linear
interpolation of $\{S_N(\theta)\}$, are \emph{asymptotically equivalent},
i.e. if either converges in distribution as $\N\to\infty$, then so
should the other, and they should have the same limit.
This is proved, for example, in  \cite{Wh}, Section 6.2.
Then one can easily extend this result to
$\{S_N(T_N^{-1})\}_{0\leq\theta\leq \mathcal{T}}$ and
$\{Y_N\}_{0\leq\theta\leq \mathcal{T}}$, since the waiting time
between two jumps has finite average.

%%%%%%%%%%%%%%%%%%%%%%%%%%%%%%%%%%%%%%%%%%%%%%%%%%%%%%%%%%%%%%%%%%%%%%%%
%%%%%%%%%%%%%%%%%%%%%%%%%%%%%%%%%%%%%%%%%%%%%%%%%%%%%%%%%%%%%%%%%%%%%%%%

\end{document}